\documentclass[12pt]{amsart}
\usepackage{amssymb} 
\usepackage[mathscr]{eucal} 
\usepackage{amsmath} 
\usepackage{epsfig}
\usepackage{mathabx}
\usepackage{graphicx}
\usepackage{amscd}
\usepackage{mathrsfs}
\usepackage{hyperref}
\hypersetup{
    colorlinks=true,
    linkcolor=blue,
    filecolor=magenta,      
    urlcolor=cyan,
}
\usepackage[margin=0.8in]{geometry}

\DeclareMathOperator{\pd}{proj.dim}

\DeclareMathOperator{\LH}{LH}

\DeclareMathOperator{\spli}{spli}

\DeclareMathOperator{\GInj}{GInj}
\DeclareMathOperator{\GProj}{GProj}

\DeclareMathOperator{\hh}{H}
\DeclareMathOperator{\cd}{cd}
\DeclareMathOperator{\hd}{hd}
\DeclareMathOperator{\Gcd}{Gcd}

\DeclareMathOperator{\GFlat}{GFlat}

\DeclareMathOperator{\sfli}{sfli}

\DeclareMathOperator{\WChar}{WChar/A}
\DeclareMathOperator{\Gpd}{Gpd}
\DeclareMathOperator{\Gid}{Gid}
\DeclareMathOperator{\Ghd}{Ghd}
\DeclareMathOperator{\fgidim}{FinGid}
\DeclareMathOperator{\supr}{sup}
\DeclareMathOperator{\Gfd}{Gfd}
\DeclareMathOperator{\Findim}{FinProjDim}
\DeclareMathOperator{\ffdim}{FinFlatDim}

\DeclareMathOperator{\module}{Mod-}

\DeclareMathOperator{\Char}{Char/A}

\DeclareMathOperator{\FinGpd}{FinGpd}
\DeclareMathOperator{\FinGfd}{FinGfd}

\DeclareMathOperator{\FP}{FP}
\DeclareMathOperator{\fd}{flat.dim}

\DeclareMathOperator{\dm}{-dim}
\DeclareMathOperator{\genf}{Fin}

\DeclareMathOperator{\openfirstbracket}{(}
\DeclareMathOperator{\closefirstbracket}{)}

\theoremstyle{plain}
\newtheorem{theorem}{Theorem}[section]

\newtheorem{lemma}[theorem]{Lemma}

\newtheorem{corollary}[theorem]{Corollary}
\newtheorem{definition}[theorem]{Definition}

\theoremstyle{remark}
\newtheorem{remark}[theorem]{Remark}

\numberwithin{equation}{section}

\begin{document}

\title[Coincidence of classical and Gorenstein concepts]
{Gorenstein versions of type $\Phi$ groups and some Gcd-cd, Ghd-hd coincidences}

\author[Rudradip Biswas]
{Rudradip Biswas }
\address{Department of Mathematics\\
University of Warwick\\
Zeeman Building, Coventry CV4 7AL, UK}
\email{rudradip.biswas@warwick.ac.uk}

\author[Dimitra-Dionysia Stergiopoulou]
{Dimitra-Dionysia Stergiopoulou}
\address{Department of Mathematics\\
University of Athens\\ Athens 15784\\
Greece}
\curraddr{%
	Department of Mathematics\\
	University of Thessaly\\
	Lamia 35100\\
	Greece}
\email{dstergiop@math.uoa.gr}


\date{\today}


\begin{abstract}
{In this short note, we characterise some Gorenstein versions of the concept of a group being of type $\Phi$ as introduced by Olympia Talelli. We also generalize a different result by Talelli regarding the coincidence of the classical and the Gorenstein cohomological dimension of torsion-free groups in Kropholler's $\LH\mathscr{F}$ class.}
\end{abstract}

\maketitle

\section{Introduction}\label{s0}

Groups of type $\Phi$ were introduced by Talelli \cite{t-phi} almost twenty years ago to provide a possible algebraic characterisation of groups $G$ with a finite dimensional model for \underline{E}G. A lot has been studied about the cohomological properties of such groups, and it is also handy that we have an unconditional characterisation of type $\Phi$ groups in terms of the finiteness of a concrete invariant. In this note, we look at Gorenstein versions of this concept and establish such characterisations in terms of analogously defined (co)homological invariants (Proposition \ref{prop-char}). Additionally, we revisit the somewhat classical question of when the cohomological dimension of a group equals its Gorenstein cohomological dimension. In a separate paper \cite[Theorem 3.5.(2)]{talelli-pams}, Talelli showed that for any torsion-free group $G$ in $\LH\mathscr{F}$, with $\mathscr{F}$ denoting the class of all finite groups, the cohomological dimension of $G$ (over $\mathbb{Z}$) agrees with its Gorenstein cohomological dimension (over $\mathbb{Z}$). We generalize this result by first enlarging the ``base" in Kropholler's hierarchy from the class of all finite groups to the class of all type $\Phi$ groups over a commutative ring $A$ satisfying a small cohomological finiteness condition, and then replacing $\mathbb{Z}$ by $A$, and we then prove an analogous version regarding the coincidence of Gorenstein homological dimension and classical homological dimension over $A$ (see Theorem \ref{coincidence}). 

\section{Notations and definitions}

All our modules will be right modules. For any ring $R$, $\spli(R)$ and $\sfli(R)$ respectively denote the supremum over the projective (resp. flat) dimension of $R$-injectives. $\GFlat(R)$, $\GProj(R)$, and $\GInj(R)$ respectively denote the full subcategories of the whole module category comprising of Gorenstein flat $R$-modules, Gorenstein projective $R$-modules, and Gorenstein injective $R$-modules. As expected, for any $R$-module, $\Gpd_R(M)$, $\Gfd_R(M)$, $\Gid_R(M)$ denote its Gorenstein projective dimension, Gorenstein flat dimension, and Gorenstein injective dimension respectively. For any (resolving or coresolving) subcategory $\mathcal{C} \subseteq \module R$, $\mathcal{C}\dm_R(M)$ denotes the minimum length of a resolution (or a coresolution, depending on the context) of $M$ with modules from $\mathcal{C}$. So, when $\mathcal{C}=\GProj(R)$ or $\GInj(R)$, $\mathcal{C}$-dimension of a module is just its Gorenstein projective or Gorenstein injective dimension respectively. $\genf \mathcal{C}\dm (R)$ denotes the ``Finitistic" $\mathcal{C}$-dimension of $R$, i.e. the supremum over all $R$-modules that have finite $\mathcal{C}$-dimension.

When $R$ is a group algebra $AG$ with $A$ is a commutative ring and $G$ any group, $\kappa_{\mathcal{C}\dm}(AG;\mathscr{X})$, for any class of subgroups $\mathscr{X}$ of $G$, is the notation for $\supr_{M \in \module AG}\{ \mathcal{C}\dm_{AG}(M): \mathcal{C}\dm_{AH}(M)<\infty$ for all $H \in \mathscr{X}\}$. For many of these dimensions, studying them over group algebras over finite subgroups of an infinite group is an easier task than studying over the infinite group directly, and that was, in part, the original motivation behind the introduction of these invariants in \cite{ck98} for $\mathcal{C}=AG$-projectives, and $\mathscr{X}=$ the class of finite subgroups of an arbitrary infinite group $G$ - a key use of this invariant in the literature for $A=\mathbb{Z}$ and any infinite group $G$ (i.e. $\kappa_{\pd}(\mathbb{Z}G;\mathscr{F})$, with $\mathscr{F}$ denoting the class of all finite groups) is the formulation of the conjecture (and the evidence towards it) that $\kappa_{\pd}(\mathbb{Z}G;\mathscr{F})<\infty \Leftrightarrow G$ admits a finite-dimensional classifying space of proper actions (see \cite[Conjecture A, Propositions 2.2 and 3.1]{t-phi}). $G$ is said to be of type $\Phi(\mathcal{C}\dm, \mathscr{X})$ if, for any $AG$-module $M$, $\mathcal{C}\dm_{AG}(M)<\infty \Leftrightarrow \mathcal{C}\dm_{AH}(M)<\infty$ for all $H \leqslant G$ such that $H \in \mathscr{X}$.

\begin{definition}
\begin{itemize}
    \item[i.] \textbf{(Kropholler's hierarchy)} For any family of groups $\mathscr{X}$, denote $\hh_0\mathscr{X}:=\mathscr{X}$, and for any successor ordinal $\alpha$, $G$ is said to be in $\hh_{\alpha}\mathscr{X}$ iff there is a finte-dimensional contractible CW-complex on which $G$ acts with stabilizers in $\hh_{\alpha-1}\mathscr{X}$. When $\alpha$ is a limit ordinal, $\hh_{\alpha}\mathscr{X}:=\bigcup_{\beta<\alpha}\hh_{\beta}\mathscr{X}$, and $G$ is said to be in $\hh\mathscr{X}$ if $G\in \hh_{\alpha}\mathscr{X}$ for some ordinal $\alpha$. Finally, we say $G \in \LH\mathscr{X}$ if every finitely generated subgroup of $G$ is in $\hh\mathscr{X}$.
    \item[ii.] \textbf{(Groups with (weak) characteristic modules)} Let $A$ be a commutative ring. The class of groups, $\mathscr{X}_{\Char}$ (groups admitting a characteristic module over $A$), is defined as the class of those groups $G$ such that there is an $A$-projective $AG$-module $M$ of finite projective dimension over $AG$ with an $A$-split $AG$-monomorphism $A \hookrightarrow M$. Similarly, the class $\mathscr{X}_{\WChar}$ (groups admitting a weak characteristic module over $A$) is the class of those groups $G$ such that there is an $A$-flat $AG$-module $N$ of finite flat dimension over $AG$ with an $A$-pure $AG$-monomorphism $A \hookrightarrow N$. 

\end{itemize}
    
\end{definition}

Since the nineties when Kropholler's hierarchy was first introduced in \cite{krop94}, a lot of progress has been made in understanding a wide range of cohomological behaviour of infinite groups lying in $\LH\mathscr{F}$, and it is also known that the classes $\{\hh_n\mathscr{F}\}_{n \in \mathbb{N}}$ keep getting strictly bigger as $n$ increases \cite{jkl}. In our paper, whenever we make our groups belong to Kropholler's hierarchy, the base class of groups is a much larger class than the class of all finite groups. 

The motivation behind the language of (weak) characteristic modules is that for any fixed commutative ring $A$ and many infinite groups $G$ satisfying some nice cohomological properties, the $AG$-module $B(G,A)$, which is given by the $G \rightarrow A$ functions of finite range, can be shown to satisfy the properties that we are demanding our characteristic or weak characteristic modules to satisfy. The term ``characteristic module" is from \cite{tal-char}, and ``weak" characteristic modules were first studied in \cite{dds}.

\section{New results}

We first start by putting together some results, mostly recent from the work of the second author, that will be quite useful for us.

\begin{theorem}\label{char-wchar} Let $G$ be a group and let $A$ be a commutative ring.
\begin{itemize}

\item[i.]  $\supr\{\Gpd_{AG}(M): M \in \module AG \}= \supr\{\Gid_{AG}(M): M \in \module AG \}=\spli(AG)$ $\openfirstbracket$this can be seen combining \cite[Theorem 4.1]{e} and \cite[Corollary 24]{de}$\closefirstbracket$, and $\supr\{\Gfd_{AG}(M):M\in \module AG\}=\sfli(AG)$ $\openfirstbracket$follows from \cite[Theorem 5.3]{e}$\closefirstbracket$.

\item[ii.] For $\spli(A)<\infty$, we have $G\in \mathscr{X}_{\Char} \Longleftrightarrow \spli(AG)<\infty$ - \cite[Theorem 2.14]{dds-ja}.

\item[iii.] For $\sfli(A)<\infty$, we have $G\in \mathscr{X}_{\WChar}\Longleftrightarrow \sfli(AG)<\infty$ - \cite[Theorem 3.14]{ks}.

\end{itemize}

\end{theorem}

The following can be seen as an easy consequence of Theorem \ref{char-wchar}.

\begin{corollary}\label{k-spli} Let $G$ be a group and let $A$ be a commutative ring. 

\begin{itemize}
\item[i.] If $\spli(A)< \infty$, then $\kappa_{\Gpd}(AG;\mathscr{X})=\spli(AG) = \kappa_{\Gid}(AG;\mathscr{X})$ for any $\mathscr{X}\subseteq \mathscr{X}_{\Char}$, and if $G \in \LH\mathscr{X}_{\Char}$, then $\kappa_{\Gpd}(AG;\mathscr{X}_{\Char}) = \kappa_{\pd}(AG;\mathscr{X}_{\Char})$.

\item[ii.] If $\sfli(A)<\infty$, then $\kappa_{\Gfd}(AG;\mathscr{X})=\sfli(AG)$ for any $\mathscr{X}\subseteq \mathscr{X}_{\WChar}$, and if $G \in \LH\mathscr{X}_{\WChar}$, then $\kappa_{\Gfd}(AG;\mathscr{X}_{\WChar}) = \kappa_{\fd}(AG;\mathscr{X}_{\WChar})$.

\end{itemize}
\end{corollary}

\begin{proof} \begin{itemize}

\item[i.] Theorem \ref{char-wchar} implies that for $\spli(A)<\infty$ and $\mathscr{X}\subseteq \mathscr{X}_{\Char}$, $\kappa_{\Gpd}(AG;\mathscr{X})=\supr\{\Gpd_{AG}(M): M \in \module AG \}= \supr\{\Gid_{AG}(M): M \in \module AG \}=\kappa_{\Gid}(AG;\mathscr{X})$ (more precisely, using Theorem \ref{char-wchar}.i-ii, we see that all $AG$-modules, when restricted to subgroups in $\mathscr{X}_{\Char}$, have finite $\Gpd$ and finite $\Gid$), and so they are all equal to $\spli(AG)$. When $G \in \LH\mathscr{X}_{\Char}$, we can use \cite[Remark 6.9.ii]{dds} to deduce $\kappa_{\pd}(AG;\mathscr{X}_{\Char}) = \spli(AG)$, so we are done. 

\item[ii.] When $\sfli(A)<\infty$, using Theorem \ref{char-wchar}.i and Theorem \ref{char-wchar}.iii, we see that all $AG$-modules, when restricted to subgroups in $\mathscr{X}_{\WChar}$, have finite $\Gfd$. Now, by Theorem \ref{char-wchar}.i., we have $\kappa_{\Gfd}(AG;\mathscr{X})=\supr\{\Gfd_{AG}(M):M\in \module AG\} = \sfli(AG)$ for any $\mathscr{X}\subseteq \mathscr{X}_{\WChar}$. When $G \in \LH\mathscr{X}_{\WChar}$, we can use \cite[Remark 6.9.i]{dds} to deduce $\kappa_{\fd}(AG;\mathscr{X}_{\WChar}) = \sfli(AG)$, so we are done. 
\end{itemize} 

\end{proof}

Now, we prove our main characterisation result. 

\begin{theorem}\label{prop-char} Let $G$ be a group, let $A$ be a commutative ring. First, let $\spli(A)<\infty$, and $\mathscr{X}\subseteq \mathscr{X}_{\Char}$ be a class of groups . Then, 

\begin{itemize}

\item[i.] $\kappa_{\Gpd}(AG;\mathscr{X})<\infty \Leftrightarrow G$ is of type $\Phi(\Gpd, \mathscr{X})$ over $A$.

\item[ii.] $\kappa_{\Gid}(AG;\mathscr{X})<\infty \Leftrightarrow G$ is of type $\Phi(\Gid, \mathscr{X})$ over $A$.

\end{itemize}

Now, let $\sfli(A)<\infty$ and $\mathscr{X}\subseteq \mathscr{X}_{\WChar}$. Then,

\begin{itemize}
\item[iii.] $\kappa_{\Gfd}(AG;\mathscr{X})<\infty \Leftrightarrow G$ is of type $\Phi(\Gfd, \mathscr{X})$ over $A$.
\end{itemize}

\end{theorem}
\begin{proof} The ``$\Rightarrow$" direction in $(i)-(iii)$ is easy because Gorenstein projectivity (resp. Gorenstein injectivity and Gorenstein flatness) is preserved under restriction to subgroups in $\mathscr{X}_{\Char}$ (resp. $\mathscr{X}_{\Char}$ and $\mathscr{X}_{\WChar}$ - see \cite[Corollaries 4.11, 5.7 and 3.7]{dds}).

\textbf{Main principle for the other direction:} ``Let $\mathcal{C}$ be a full subcategory of $\module AG$ that is either $\GProj(AG)$, $\GFlat(AG)$ or $\GInj(AG)$. Let $\mathscr{X}$ be a class of subgroups of $G$ such that, for any $H \in \mathscr{X}$, $\supr \{\mathcal{C}\dm_{AH}(M): M \in \module AH\}$ is finite (\textbf{Hypothesis H}). Now, if $\genf\mathcal{C}\dm(AG)$ is not finite, then $\exists M \in \module AG$ such that $\mathcal{C}\dm_{AG}(M)$ is not finite but $\mathcal{C}\dm_{AH}(M)$ is finite for all $H\in \mathscr{X}$. This is easy to see because if $\genf\mathcal{C}\dm(AG)$ is not finite, then for any $n$, we can take an $AG$-module $M_n$ such that $n \leq \mathcal{C}\dm_{AG}(M_n)<\infty$, and then we can choose $M$ to be $\bigoplus_n M_n$ for $\mathcal{C}=\GProj(AG)$ or $\GFlat(AG)$ and $\prod_n M_n$ for $\mathcal{C}=\GInj(AG)$. It is clear that $\mathcal{C}\dm_{AG}(M)$ is not finite, whereas for any $H\in \mathscr{X}$, $\mathcal{C}\dm_{AH}(M)\leq \supr_n \mathcal{C}\dm_{AH}(M_n)<\infty$ (by Hypothesis H)."

For the ``$\Leftarrow$" direction in $(i)-(iii)$, first note that when $G$ is of type $\Phi(\Gpd,\mathscr{X})$ (resp. type $\Phi(\Gid,\mathscr{X})$ or $\Phi(\Gfd,\mathscr{X})$) over $A$, $\kappa_{\Gpd}(AG;\mathscr{X}) = \FinGpd(AG)$ (resp. $\kappa_{\Gid}(AG;\mathscr{X}) = \fgidim(AG)$ and $\kappa_{\Gfd}(AG;\mathscr{X}) = \FinGfd(AG)$). Now, to use the above ``Main Principle" and finish the proof, we just have to check that for the $A$ and $\mathscr{X}$ in $(i)-(iii)$, \textbf{Hypothesis H} is satisfied for $\mathcal{C}=\GProj(AG)$, $\GInj(AG)$ and $\GFlat(AG)$ respectively, and that is checked easily by Theorem \ref{char-wchar}.

\end{proof}

Combining Theorem \ref{prop-char} with Theorem \ref{char-wchar} and Corollary \ref{k-spli}, we get the following:
\begin{corollary} Let $A$ be a commutative ring. 

\begin{itemize}
\item[i.] If $\spli(A)<\infty$, then $\{$Groups of type $\Phi(\Gpd,\mathscr{X})$ over $A\}=\{$Groups of type $\Phi(\Gid,\mathscr{X})$ over $A\}=\mathscr{X}_{\Char}$, for any $\mathscr{X}\subseteq \mathscr{X}_{\Char}$. So, taking $\mathscr{X}=\mathscr{F}$ or $\mathscr{F}_{\phi,A}$ or all of $\mathscr{X}_{\Char}$ does not change the class.

\item[ii.] If $\sfli(A)<\infty$, then $\mathscr{X}_{\WChar}=\{$Groups of type $\Phi(\Gfd,\mathscr{X})$ over $A\}$, for any $\mathscr{X}\subseteq \mathscr{X}_{\WChar}$. Again, taking $\mathscr{X}=\mathscr{F},\mathscr{F}_{\phi,A},\mathscr{X}_{\Char}$ or all of $\mathscr{X}_{\WChar}$ does not change the class.

\end{itemize}
\end{corollary}

\begin{remark} It is easy to see and has been known for a while that for any group $G$ and any commutative ring $A$ of finite global dimension, $\kappa_{\pd}(AG;\mathscr{F})<\infty$ iff $G$ is classically of type $\Phi$ over $A$ (for a precise reference, see \cite[Lemma 4.1]{b1}). Theorem \ref{prop-char} was therefore very much expected as Gorenstein analogues of this classical type $\Phi$ characterisation. However, it is still open as to whether $\spli(AG)<\infty \Leftrightarrow G$ is (classically) of type $\Phi$ over $A$ (we are still sticking to $A$ having finite global dimension); only the ``$\Leftarrow$" direction is known to be true unconditionally. In this context, it is interesting to note in light of Corollary \ref{k-spli} and Theorem \ref{prop-char} that any Gorenstein analogue of type $\Phi$ behaviour is equivalent to the group algebra having finite $\spli$ or $\sfli$.

\end{remark}

\begin{lemma}\label{69} 
 These two results have been proved for $\mathscr{X}=\mathscr{X}_{\Char}$ and $\mathscr{X} = \mathscr{X}_{\WChar}$ respectively in \cite[Remark 6.9]{dds}, and the same proofs work here.
 \begin{itemize}
 	\item[i.] If $\spli(A)<\infty$, then for any $\mathscr{X}\subseteq \mathscr{X}_{\Char}$ and $G \in \LH\mathscr{X}$, $\kappa_{\pd}(AG;\mathscr{X}) = \spli(AG) = \Findim(AG)$.
 
 \item[ii.] If $\sfli(A)<\infty$, then for any $\mathscr{X}\subseteq \mathscr{X}_{\WChar}$ and $G\in \LH\mathscr{X}$, $\kappa_{\fd}(AG;\mathscr{X}) = \sfli(AG) = \ffdim(AG)$.
\end{itemize}

\end{lemma}

Now, we can prove our final promised result generalizing \cite[Theorem 3.5]{talelli-pams}.

\begin{theorem}\label{coincidence} Let $A$ be a commutative ring and let $G$ be a torsion-free group. Let the class of classical type $\Phi$ groups, i.e. groups of type $\Phi(\pd,\mathscr{F})$ over $A$, be denoted $\mathscr{F}_{\phi,A}$. Then, 
\begin{itemize}
\item[i.] For $\spli(A)<\infty$, $\Gcd_A(G)=\cd_A(G)$ as long as $G \in \LH\mathscr{F}_{\phi,A}$. 
\item[ii.] For $\sfli(A)<\infty$, $\Ghd_A(G)=\hd_A(G)$ as long as $G \in \LH\mathscr{F}_{\phi,A}$.
\end{itemize}
\end{theorem}

\begin{proof}
\begin{itemize} \item[i.] If $\Gcd_A(G)$ is not finite, $\cd_A(G)$ cannot be finite, and when $\Gcd_A(G)<\infty$, we will be done if we can show that $\cd_A(G)<\infty$. Note that when $\Gcd_A(G)<\infty$,  $\spli(AG)<\infty$ by \cite[Corollary 2.18.ii]{dds-ja}. So, by Lemma \ref{69}.i., $\kappa_{\pd}(AG;\mathscr{F}_{\phi,A})<\infty$. Now, the $AG$-module $A$ has finite projective dimension over all finite subgroups of $G$ (as $G$ is torsion-free), and so it also has finite projective dimension over all $\mathscr{F}_{\phi,A}$-subgroups of $G$, which means $\cd_A(G)<\infty$ as $\kappa_{\pd}(AG;\mathscr{F}_{\phi,A})<\infty$, and we are done.

\item[ii.] This is proved in the same way as (i) by using Lemma \ref{69}.ii instead of Lemma \ref{69}.i, and \cite[Corollary 6.2]{ks} instead of \cite[Corollary 2.18.ii]{dds-ja}.

\end{itemize}
\end{proof}

\noindent \textit{Acknowledgements.}  The research project is implemented in the framework of  H.F.R.I. Call “Basic research Financing Horizontal support of all Sciences” under the National Recovery and Resilience Plan “Greece 2.0” funded by the European Union Next Generation EU, H.F.R.I. Project Number: 14907.  
\begin{center}
	\includegraphics[scale=0.8]{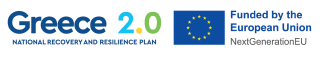}
	\hskip 1cm
	\includegraphics[scale=0.05]{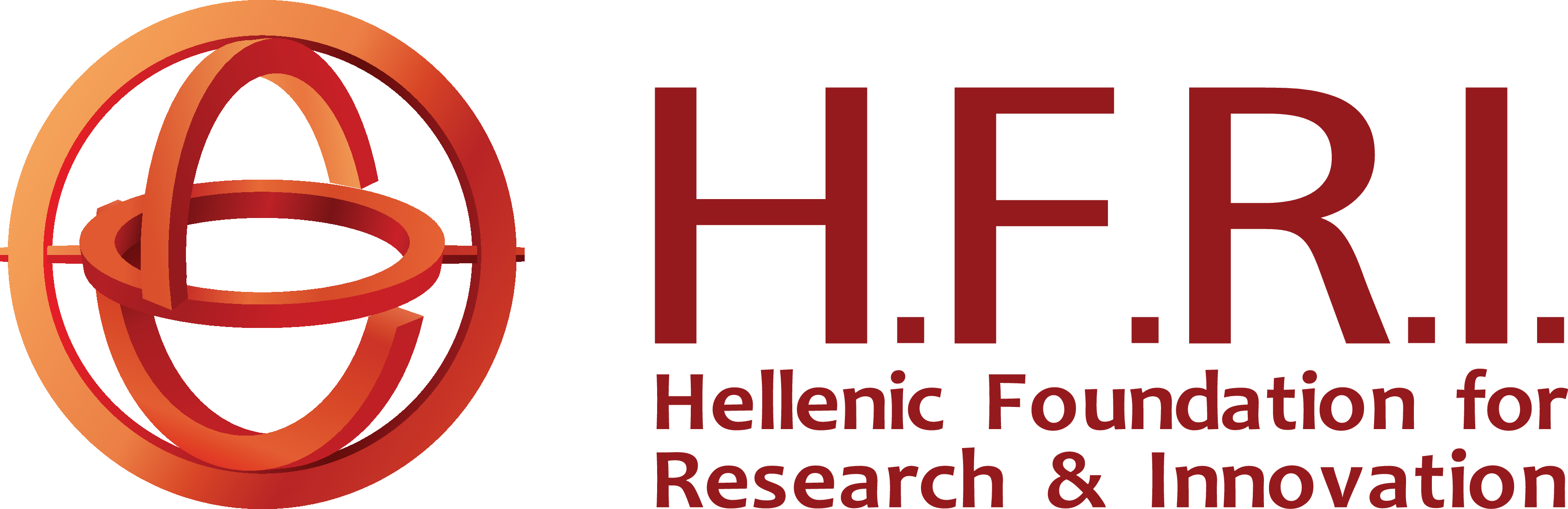}
\end{center}

\end{document}